\documentclass[11pt,oneside,english]{amsart}
\usepackage[T1]{fontenc}
\usepackage[utf8]{inputenc}
\usepackage{color}
\usepackage{babel}
\usepackage{amstext}
\usepackage{amsthm}
\usepackage{amssymb}
\usepackage{graphicx}
\usepackage[caption=false]{subfig}
\usepackage[unicode=true,pdfusetitle,
 bookmarks=true,bookmarksnumbered=false,bookmarksopen=false,
 breaklinks=false,pdfborder={0 0 0},backref=false,colorlinks=true]
 {hyperref}
\hypersetup{
 citecolor=blue}

\DeclareMathOperator{\supp}{supp}

\DeclareMathOperator{\Mather}{Mather}

\makeatletter
\numberwithin{equation}{section}
\numberwithin{figure}{section}
\theoremstyle{plain}
\newtheorem*{cor*}{\protect\corollaryname}
\theoremstyle{plain}
\newtheorem{thm}{\protect\theoremname}[section]
\theoremstyle{definition}
\newtheorem{defn}[thm]{\protect\definitionname}
\theoremstyle{question}

\theoremstyle{remark}
\newtheorem{rem}[thm]{\protect\remarkname}
\theoremstyle{plain}
\newtheorem{prop}[thm]{\protect\propositionname}
\theoremstyle{plain}
\newtheorem{lem}[thm]{\protect\lemmaname}
\theoremstyle{plain}
\newtheorem{cor}[thm]{\protect\corollaryname}

\newtheorem{ex}[thm]{Example}

\newcommand{\arcangle}{%
  \mathord{<\mspace{-9mu}\mathrel{)}\mspace{2mu}}%
}

\usepackage{geometry}
\usepackage{amsmath}

\numberwithin{equation}{section}
\numberwithin{figure}{section}
\usepackage{enumitem}		
 \let\footnote=\endnote
\@ifundefined{lettrine}{\usepackage{lettrine}}{}

\theoremstyle{definition}

\def\a{\alpha}

\def\R{\mathbb{R}}

\def\N{\mathbb{N}}
\def\Z{\mathbb{Z}}

\usepackage{float}


\subjclass[2000]{}

\def\a{\alpha}

\def\R{\mathbb{R}}

\def\sl2{\text{SL}(2,\R)}

\def\A{\mathcal{A}}
\def\B{\mathcal{B}}

\def\Sig{\Sigma}
\def\glr{\text{GL}_d(\R)}
\def\gl2{\text{GL}_2(\R)}

\makeatother

  \providecommand{\corollaryname}{Corollary}
  \providecommand{\definitionname}{Definition}
  \providecommand{\lemmaname}{Lemma}
  \providecommand{\propositionname}{Proposition}
  \providecommand{\remarkname}{Remark}
  \providecommand{\theoremname}{Theorem}
\providecommand{\theoremname}{Theorem}

\usepackage{tikz,xcolor,hyperref}

\definecolor{lime}{HTML}{A6CE39}
\DeclareRobustCommand{\orcidicon}{
	\begin{tikzpicture}
	\draw[lime, fill=lime] (0,0) 
	circle [radius=0.16] 
	node[white] {{\fontfamily{qag}\selectfont \tiny ID}};
	\draw[white, fill=white] (-0.0625,0.095) 
	circle [radius=0.007];
	\end{tikzpicture}
	\hspace{-2mm}
}
\foreach \x in {A, ..., Z}{\expandafter\xdef\csname orcid\x\endcsname{\noexpand\href{https://orcid.org/\csname orcidauthor\x\endcsname}
			{\noexpand\orcidicon}}
}

\author[Reza Mohammadpour]{Reza Mohammadpour\orcidA{}}

\address{Department of Mathematics, Uppsala University, Box 480, SE-75106, Uppsala, Sweden.}
\date{\today}

\email{reza.mohammadpour@math.uu.se}
\begin{document}
\title{Entropy of Lyapunov maximizing measures of $SL(2,\R)$ typical cocycles}

\maketitle

\begin{abstract}
In this paper we study ergodic optimization problems for typical cocycles. We consider one-step $SL(2,\R)$-cocycles that satisfy pinching and twisting conditions (in the sense of \cite
{BV}). We prove that the Lyapunov maximizing measures have zero entropy under additional assumptions that the maps $e_1$ and $e_2$ are one-to-one on the Mather set.
\end{abstract}

\section{Introduction}
Our main subject is {\it ergodic optimization}. The large-scale image of ergodic optimization is that one is interested in optimizing potential functions over the (typically externally complex) class of invariant measures for a dynamical system; see \cite{Je, Bo}.The field also has both local aspects in which the optimization is studied on individual functions and has general aspects in which the optimization is considered in the large on whole Banach spaces. There has been input from physicists with numerical simulations suggesting that the optimizing measures are typically supported on periodic orbits, which is the main question of this subject. That is converted to the dynamical system language as follows, for hyperbolic base dynamics and for typical functions, optimizing orbits should have low complexity. This is proved by Contreras \cite{C}, who showed that the optimizing orbits with respect to generic H\"older/Lipschitz potentials over an expanding base are periodic.

In this paper, we would like to study ergodic optimization in a non-commutative setting where {\it matrix cocycles} are a well-known example of a noncommutative system. In the {\it ergodic optimization of Lyapunov exponents}, the quantities we want to maximize are the associated Lyapunov exponents of matrix cocycles. We would like to investigate the low complexity phenomena mentioned above in the matrix cocycle case. For 2-dime\-n\-sional one-step cocycles, Bochi and Rams \cite{BR} showed that
Lyapunov-maximizing measures have zero entropy. Their
sufficient conditions for zero entropy are domination and existence of strictly invariant
families of cones satisfying a non-overlapping condition, which is open, but not a typical assumption. Jenkinson and Pollicott \cite{JP}  also investigated the zero entropy of Lyapunov maximizing measures under the almost same assumption, but their techniques are different. Bousch and Mairesse \cite{BMa} showed  that the maximizing products are not always periodic, thus disproving the so-called \textit{Finiteness  Conjecture}. Their examples are difficult to construct, and are broadly believed to be rare. The complexity of the matter already appears in the simple setting one-step cocycles. Indeed, such sets appear in the literature both as finiteness counterexamples \cite{BTV, HMST, O}.

In this paper, we deal  with  $SL(2,\R)$ one-step  cocycles.  The smaller size of this body of research on the entropy of Lyapunov optimization measures can perhaps be explained
by the fact that the behavior of the Lyapunov exponents of matrix cocycles is much more complicated than the behavior of Birkhoff averages. Inspired by \cite{BR}, we show that the entropy of Lyapunov maximizing measures is zero for typical cocycles under some assumptions on the Mather set.

 This should be also contrasted with the situation typically found in the {\it thermodynamic formalism} of ergodic theory, in which the measures picked out by variational principles tend to have wide support. Ergodic optimization may be seen as the {\it zero-temperature} limit of thermodynamic formalism; see \cite{Bremont, CH-zerotem10, Jenkinson-MU, Morris-zero, Moh20a}.

\subsection{Ergodic optimization of Birkhoff averages} We call $(X, T)$ a \textit{topological dynamical system} (TDS) if $T : X \rightarrow X$ is
a continuous map on the compact metric space $X$. We denote by $\mathcal{M}(X, T)$ the space of all $T$-invariant Borel probability
measures on $X$, which the space is a non-empty convex set and is
compact with respect to the weak* topology. Let $\mathcal{E}(X,T)$ be the subset formed by ergodic measures, which are exactly the extremal points of  $\mathcal{M}(X, T)$.

Let $f:X\rightarrow \R$ be a continuous function. We denote by $S_{n}f(x):=\sum_{k=0}^{n-1}f(T^{k}(x))$ the \textit{Birkhoff sum}, and we call $\lim_{n\rightarrow \infty} \frac{1}{n}S_{n}f(x),$ the {\it Birkhoff average} and $\mu$-almost every $x\in X$,  it is well-defined. We denote by $\beta(f)$ and $\alpha(f)$ the supremum and infimum of the Birkhoff average over $x\in X$, respectively; we call these numbers the \textit{maximal and minimal ergodic averages of $f$}. Since $\alpha(f)=-\beta(-f)$, let us focus the discussion on the quantity $\beta$ that can also be characterized as
\[
\beta(f)=\sup _{\mu \in \mathcal{M}(X, T)} \int f d \mu.
\]
By compactness of $\mathcal{M}(X,T)$, the supremum is attained; such measures will be called \textit{Birkhoff maximizing measures},  we denote them by $\mathcal{M}_{max}(f).$

Given a continuous function $f: X \rightarrow \mathbb{R}$ with some prescribed regularity, under suitable dynamical hypotheses there exists a continuous function $g: X \rightarrow \mathbb{R}$ with the property that $f \leq g \circ T-g+\beta(f)$. This relation is equivalent to the statement that there exists a continuous $g$ such that $\sup (f+g-g \circ T)=\beta(f)$. Results of this type are sometimes called \textit{Mañé Lemma}; see
\cite{Bousch, Bousch1, CLT, Savchenko} for various versions and approaches, and see \cite[Proposition 2.1]{Bo} for a
negative result.

In the ergodic optimization of Birkhoff averages, a {\it maximizing set} is a closed subset such that an invariant probability measure is a maximizing measure if and only if its support lies on this subset. If Mañé Lemma holds, then the existence of such sets is guaranteed in any setting. The smallest maximizing set is the {\it Mather set} that is defined as the union of the supports of all maximizing measures, which is borrowed from Lagrangian dynamics; see more information \cite{CLT, Je}.

\subsection{Ergodic optimization of Lyapunov exponents}
We assume that $X$ is a compact metric space and $T:X\to X$ is a homeomorphism.  Let $\A:X \rightarrow GL(d, \R)$ be  a continuous map. We define a \textit{matrix cocycle} $F:X \times \R^{d} \rightarrow X\times \R^{d}$  as
\[ F(x, v)=(T(x), \A(x)v).\]
We say that $F$ is generated by $T$ and $\A$; we will also denote it by $(\A, T).$ Observe $F^{n}(x, v)=(T^{n}(x), \A^{n}(x)v)$ for each $n \geq 1$, where
\begin{equation}\label{product}
\A^{n}(x)=\A(T^{n-1}(x))\A(T^{n-2}(x))\cdots \A(x).
\end{equation}

By Kingman’s subadditive ergodic theorem, for any $\mu \in \mathcal{M}(X, T)$, the \textit{top Lyapunov exponent} exists:
\begin{equation}\label{LE}
\chi(x, \A):= \lim_{n\rightarrow \infty}\frac{1}{n}\log \|\A^{n}(x)\|,
\end{equation}
 for $\mu$-almost every $x \in X.$
 Let us denote $\chi(\mu, \A)=\int \chi(. , \A )d\mu.$ If the measure $\mu$ is ergodic then $\chi(x, \A)=\chi(\mu, \A)$ for $\mu$-almost every $x\in X.$

Similar to what we did for Birkhoff averages, we can either maximize or minimize the top Lyapunov exponent \eqref{LE}; we define the maximal and minimal Lyapunov exponent as follows
 \begin{equation}\label{max1}
\beta(\A):= \limsup_{n\rightarrow \infty}\frac{1}{n}\log \sup_{x\in X}\|\A^{n}(x)\|,
\end{equation}
\begin{equation}
\alpha(\A):= \liminf_{n\rightarrow \infty}\frac{1}{n}\log \inf_{x\in X}\|\A^{n}(x)\|.
\end{equation}
It is easy to see that the limit in \eqref{max1} exists (see \cite{Mor13}). Moreover, the maximal and minimal Lyapunov exponent can be characterized as the supremum and infimum of the Lyapunov exponents of measures over invariant measures, respectively, i.e. 
\begin{equation}\label{max2}
\beta(\A)=\sup_{\mu \in \mathcal{M}(X,T)}\chi(\mu, \A),
\end{equation}
\begin{equation}\label{min2}
\alpha(\A)=\inf_{\mu \in \mathcal{M}(X,T)}\chi(\mu, \A).
\end{equation}
In \eqref{max2}, the supremum is always attained by an ergodic measure. This follows from the fact that $\mathcal{M}(X, T)$ is a compact convex set  whose extreme points are exactly the ergodic measure, and upper semi continuity of $\chi(., \A)$ with respect to the weak* topology (see \cite{Mor13, BR}). On the other hand, the infimum in \eqref{min2} does not necessarily attend; see \cite{BMa, Moh20} for more information.

 Let us define the set of \textit{Lyapunov maximizing measures} of the cocycle $\A$ to be the set of invariant measures on $X$ given by 
\[ \mathcal{M}_{max}(\A):=\{\mu \in \mathcal{M}(X, T): \hspace{0.1cm} \beta(\A)=\chi(\mu, \A) \}. \]
The $\mathcal{M}_{max}$ is non-empty, compact and convex. Another significant strand of research in the ergodic optimization of Lyapunov exponents, again already present in early works, was its interpretation (see \cite{Moh20a}) as a limiting \textit{zero temperature} version of the subadditive thermodynamic formalism, with maximizing measures (referred to as \textit{ground states} by physicists) arising as zero temperature accumulation points of equilibrium measures of the subadditive potentials; working this area has primarily focused on understanding convergence and non-convergence in the zero temperature limit.

Following Morris \cite{Mor13}, we define the Mather set $\Mather(\Phi_{\A})$ for the subadditive potential $\Phi_{\A}:=\{\log \|\A^{n}\|\}_{n=1}^{\infty}$ as the union of the supports of all Lyapunov maximizing measures, i.e., 
\[ \Mather(\Phi_{\A})=\bigcup_{\mu \in \mathcal{M}_{max}(\A)}\supp \mu.\]
We denote $\mathcal{K} =\Mather(\Phi_{\A})$ to simplify the notations in the proofs. The Mather set is a nonempty, compact, and $T$-invariant set (see \cite{Mor13, BGar}).

\subsection{Typical cocycles}

Let $\Sigma=\{1, \ldots, k\}^{\mathbb{Z}}$ be the collection of all infinite words obtained from integers $\{1, \ldots, k\}$.  The empty word $\left.i\right|_{0}$ is denoted by $\varnothing$. We denote by $\Sigma_{n}$ the set of words with the length $n$ and $\Sigma_{*}=\bigcup_{n \in \mathbb{N}} \Sigma_{n} \cup\{\varnothing\}$. Thus $\Sigma_{*}$ is the collection of all finite words.

A well-known example of matrix cocycles is \textit{one-step cocycles} which is defined as follows. Assume that $\Sig=\{1,...,k\}^{\Z}$ is a symbolic space. Suppose that $T:\Sig \rightarrow \Sig$  is a shift map, i.e. $T(x_{l})_{l\in \Z}=(x_{l+1})_{l\in \Z}$. Given a $k$-tuple of matrices $\textbf{A}=(A_{1},\ldots,A_{k})\in \glr^{k}$, we associate with it the locally constant map $\mathcal{A}:\Sig \rightarrow GL(d, \R)$ given by $\mathcal{A}(x)=A_{x_{0}},$ that means the matrix cocycle $\mathcal{A}$ depends only on the zero-th symbol $x_0$ of $(x_{l})_{l\in \Z}$. In this case, we say that $(\mathcal{A}, T)$ is a one-step cocycle; when the context is clear, we say that $\mathcal{A}$ is a one-step cocycle. For any length $n$ word $I=i_{0} \ldots i_{n-1},$  we denote 
\[\mathcal{A}_{I}:=A_{i_{n-1}}\ldots A_{i_{0}}.\]
Therefore, when $\A$ is a one-step cocycle, \[ \A^n(x)=A_{x_{n-1}}\ldots A_{x_{0}}.\]

For  one-step  cocycles,  the  value of the maximal Lyapunov exponent can  be alternatively defined as follows
\[ \beta(\A):=\lim_{n\to \infty}\frac{1}{n}\log \sup_{i_{1},
\ldots,i_{n}}\|A_{i_{n}} \ldots A_{i_{1}}\|.\]

For $A \in \gl2$, we denote by $\sigma_1(A) \geq \sigma_2
(A)$ the singular values of $A$, that is, the eigenvalues of $\left(A^* A\right)^{1 / 2}$.

\begin{defn} We say that a one-step cocycle $\mathcal{A}:\Sigma \to SL(2, \R)$
\begin{itemize}
\item is \textit{pinching} if for   any constant $\kappa > 1$ there is $I \in \Sigma_{*}$ such that $\frac{\sigma_{1}(\A_{I})}{\sigma_{2}(\A_{I})} > \kappa$;
\item  is \textit{twisting} if given any vector lines  $G, F_1, \ldots, F_n \subset \R^2$, there exists $J \in \Sigma_{*}$ such that $\mathcal{A}_{J} \left(G\right) \notin \{F_1, \ldots, F_{n}\}.$
\end{itemize}
\end{defn}
 
 We say that the cocycle $\mathcal{A}$ is \textit{typical} if it is pinching and twisting. The above definition comes from \cite[Subsection A.4.5]{AV07}, where they explained that the above definitions of pinching and twisting are equivalent to the general definitions \cite[Definition 1.2]{AV07}. Avila and Viana \cite{AV07, AV07-acta} and Bonatti and Viana \cite{BV} showed that the set of typical cocycles is open and dense.

\subsection{Domination and Multicone}\label{multi}
Let 
$X$ be 
a compact metric space. Let $\A:X \to GL(2,\R)$ be a matrix cocycle function over a homeomorphism map $(X,T).$ We say that $2\times 2$ matrix cocycle $(\A,T)$ is \textit{dominated} or \textit{uniformly hyperbolic} with respect to $X$ if there are two continuous maps $e_{1}$ and $e_{2}$ from $X$ to $\mathbb{P}\R^2$, which are a splitting of $\R^2$ as the sum of two
one-dimensional subspaces $e_{1}(\omega), e_{2}(\omega)$ for each $\omega \in X$, and the following properties holds:
\begin{itemize}
\item equivariance:
\[\A(\omega)(e_{i}(\omega))=e_{i}(T(\omega)) \quad \text{for all } \omega \in X \text{ and }  i\in\{1,2\};\]  \item dominance: there are constants $c>0$ and $0<\delta<1$ such that
\[\frac{\left\|\A^{n}(\omega) \mid e_{1}(\omega)\right\|}{\left\|\A^{n}(\omega) \mid e_{2}(\omega)\right\|} \geqslant c e^{\delta n} \quad \text { for all } \omega \in X \text { and } n \geqslant 1.\]
\end{itemize}

By the domination properties, if $x \in X$ and $v \in \mathbb{R}^{2} \backslash e_{2}(x)$ then
\begin{equation}\label{domin}
0<\lim _{n \rightarrow \infty} \frac{\left\|\A^{n}(x) v\right\|}{\left\|\A^{n}(x) \mid e_{1}(x)\right\|}<\infty, \quad \lim _{n \rightarrow \infty} \arcangle \left(\A^{n}(x) v, e_{1}\left(T^{n} (x)\right)\right)=0,
\end{equation}
where the angle defined by a distance function on $\mathbb{P}\R^2$ (see \cite{Arn}).

Bochi and Gourmelon \cite{BG09} also proved that domination can be characterized in terms of singular values. Indeed, the cocycle $\A: X \to \mathrm{GL}(2,\mathbb{R})$ is dominated with respect to $Y$, where $Y \subset X$ is a compact and $T$-invariant set, if there are constants $C > 0$ and $0 < \tau < 1$ such that
\[
\frac{\sigma_{2}(\A^{n}(x))}{\sigma_{1}(\A^{n}(x))} < C \tau^n \hspace{0.2cm} \forall x \in Y, \forall n \in \mathbb{N}.
\]

Bochi and Garibaldi \cite{BGar} proved that a typical cocycle  is dominated with respect to the Mather set $\mathcal{K}$.

\begin{thm}\label{mather:dom}
Let $\mathcal{A}:\Sig \to SL(2, \R)$ be a one-step cocycle. Assume that $\mathcal{A}:\Sig \to SL(2,\R)$ is a typical cocycle. Then, the cocycle $\A$ is dominated with respect to the Mather set $\mathcal{K}$.
\end{thm}
\begin{proof}
It follows from a combination of \cite[Remark 3.12]{BGar}, \cite[Theorem 5.2]{BGar}, and \cite[Theorem 6.5]{BGar}.
\end{proof}

A related result was proved by Morris \cite[Theorem 2.1]{Morris2010bis}. He proved a dominated
splitting under the weaker assumption of relative product boundedness, but
with the strong hypothesis that the set is minimal.

\begin{rem}Note that Theorem \ref{mather:dom} is proved for $\text{GL}(2, \mathbb{R})$-cocycles in \cite{BGar}, but in this case, an invariant set contained in the Mather set might contain the Mather set for the second Lyapunov exponent. The reason is that \cite{BGar} considers the trivial splitting $\mathbb{E}_{\mathcal{K}} \oplus 0$ as a dominated splitting with respect to the Mather set $\mathcal{K}$. However, we cannot have this case for $\text{SL}(2, \mathbb{R})$-cocycles as we have the dominated splitting $\mathbb{E}_{\mathcal{K}} \oplus \mathbb{F}_{\mathcal{K}}$.
\end{rem}

Let $(\A, T)$ be a one-step cocycle generated by $(A_1, \ldots, A_k)\in GL(2, \R)^k$.
Let $\R_{\ast}^2:=\R^2 \textbackslash \{0\}.$ The standard symmetric cone in $\mathbb{R}_{*}^{2}$ is
$$
C_{+}:=\left\{(x, y) \in \mathbb{R}_{*}^{2} ; x y \geqslant 0\right\}.
$$
An image of $C_{+}$ by a linear isomorphism is a \textit{cone} in $\mathbb{R}_{*}^{2}$ and a \textit{multicone} in $\mathbb{R}_{*}^{2}$ is a disjoint union of finitely many cones.

We say that a multicone $M \subset \mathbb{R}_{*}^{2}$ is \textit{strictly forward-invariant} with respect to $\left(A_{1}, \ldots, A_{k}\right)$ if the image multicone $\bigcup_{i} A_{i}(M)$ is contained in the interior of $M$ (see figure \ref{fig1}).

If $M$ is a multicone, its complementary multicone $M_{\mathrm{co}}$ is defined as the closure (relative to $\mathbb{R}_{*}^{2}$ ) of $\mathbb{R}_{*}^{2} \backslash M .$ If $M$ is strictly forward-invariant with respect to $\left(A_{1}, \ldots, A_{k}\right)$ then $M_{\mathrm{co}}$ is \textit{strictly backwards-invariant}, i.e., strictly forward-invariant with respect to $\left(A_{1}^{-1}, \ldots, A_{k}^{-1}\right)$.

It was proved in \cite{ABY, BG09} that the one-step cocycle $\A$  generated by $(A_1, \ldots, A_k)\in GL(2, \R)^k$ is dominated  if and only if $\A$ has a strictly forward-invariant multicone. Assume that the one-step cocycle $\A$ generated by $(A_1, \ldots, A_k)$ is dominated with respect to $\Sigma$. Let $e_1, e_2: \Sigma \to \mathbb{P}\R^2$ be the invariant directions forming the
dominated splitting, and let $M \subset \mathbb{R}_{*}^{2}$ be a strictly forward-invariant multicone,
and let $M_{co}$ be the (strictly backwards-invariant) complementary multicone. Then, for every $x=(x_{n})_{n\in \Z} \in \Sigma,$
\begin{equation}\label{def-unst}
\{e_{1}(x)\}=\bigcap_{n=1}^{\infty}A_{x_{-1}}^{'}\ldots A_{x_{-n}}^{'}(M^{'}),
\end{equation}
\begin{equation}\label{def-st}
\{e_{2}(x)\}=\bigcap_{n=1}^{\infty}(A_{x_{n-1}}^{'}\ldots A_{x_{0}}^{'})^{-1}(M_{co}^{'}),
\end{equation}
where $M^{'}:=\{v^{'} \in \mathbb{P}\R^2; v\in M\}$ and $\A^{'}:\mathbb{P}\R^2 \to \mathbb{P}\R^2$ defined by $\A^{'}(v^{'})=(\A(v))^{'}$. 
In particular, for $x=\left(x_{i}\right)_{i \in \mathbb{Z}}, e_{1}(x)$ depends only on $x_{-}:=\left(\ldots, x_{-2}, x_{-1}\right)$ while $e_{2}(x)$ depends only on $x_{+}:=\left(x_{0}, x_{1}, \ldots\right) .$ (That is, $e_{1}$, respectively $e_{2}$, is constant on local unstable, respectively stable, manifolds.)

 We say that the one-step cocycle $\A$ generated by $(A_{1},\ldots, A_k)$ satisfies the \textit{forward nonoverlapping condition ($\mathrm{NOC}$)} on $\Sigma$ if the cocycle has a strictly forward-invariant multicone $M \subset \mathbb{R}_{*}^{2}$ such that
$$
A_{i}(M) \cap A_{j}(M)=\varnothing \text { whenever } i \neq j.
$$


It is easy to see that $e_1$, respectively $e_2$, is one-to-one on $\Sigma$ if the forward, respectively backwards,  $\mathrm{NOC}$ holds on $\Sigma$ by $\eqref{def-unst}$,  respectively  $\eqref{def-st}$.  In fact, the latter statement is both necessary and sufficient; see \cite[Proposition 2.10]{BR}. For dominated $2 \times 2$ cocycles, Bochi and Rams \cite{BR} showed that Lyapunov maximizing measures have zero entropy under the $\mathrm{NOC}$ condition on $\Sigma.$

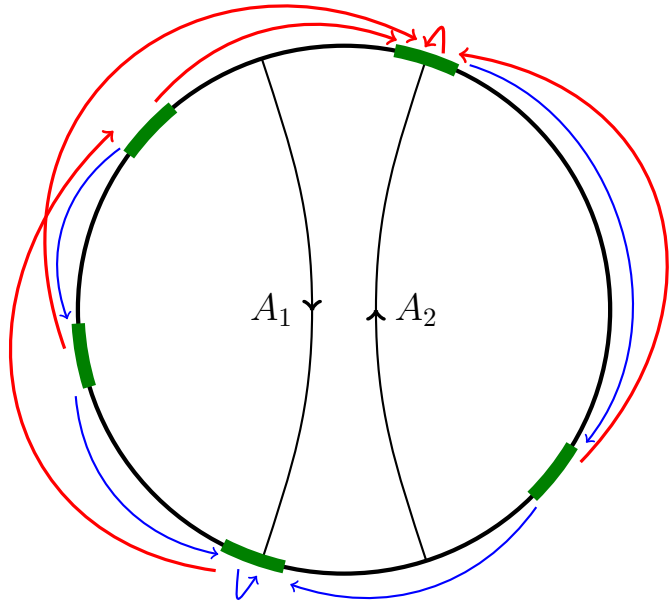
\begin{figure}[ht]
\begin{tikzpicture}[scale=0.7]

    \fill [white] (0,0) circle (7);
    
    \draw [thick] (-108:5) to [out=72,in=-90] (-0.6,0) to [out=90,in=108+180] (108:5); \draw [ultra thick, ->] (-0.6,-0.01) -- +(0,-0.01);
    \draw (-0.6,0) node [left=3pt] {\large $A_1$};
    
    \draw [thick] (-72:5) to [out=108,in=-90] (0.6,0) to [out=90,in=72+180] (72:5); \draw [ultra thick, ->] (0.6,-0.01) -- +(0,0.01);
    \draw (0.6,0) node [right=3pt] {\large $A_2$};
    
    \tikzstyle{solidarrow}=[blue, thick, ->]
    \draw [solidarrow] (63:5.2) to [bend left=56] (-29:5.2);
    \draw [solidarrow] (-46:5.2) to [bend left=37] (259:5.3);
    \draw [solidarrow] (198:5.3) to [bend right=37] (243:5.2);
    \draw [solidarrow] (144:5.2) to [bend right=37] (182:5.2);
    \draw [solidarrow] (248:5.3) .. controls (250:6) .. (252:5.3);
    
    \tikzstyle{solidarrow}=[very thick, red, ->]
    \draw [solidarrow] (69:5.2) .. controls (71:5.8) .. (73:5.2);
    \draw [solidarrow] (327:5.3) .. controls (0:7.2) and (33:7.7) .. (66:5.3);
    \draw [solidarrow] (244:5.5) .. controls (213:8) and(177:7.5) .. (142:5.5);
    \draw [solidarrow] (188:5.3) .. controls (150:8) and(110:7.6) .. (75:5.3);
    \draw [solidarrow] (132:5.3) to [bend left=35] (79:5.2);
    
    \draw [ultra thick] (0,0) circle (5);
    \foreach \a in {72,137,190,250,322} \draw [green!50!black, line width=5] (\a-7:5) arc (\a-7:\a+7:5);
    
\end{tikzpicture}
\caption{This is an example of a uniformly hyperbolic $A=(A_{1},A_{2})$ and a multicone with with 5 components that satisfies the $\mathrm{NOC}$ condition. Inner arrows indicate stable and
unstable directions of $A_{1}$ and $A_{2}$. Blue and red outer arrows indicate
the action of $A_{1}$ and $A_{2}$ in the components of the multicone, respectively.} \label{fig1}
\end{figure}

 \subsection{Precise setting and statements}
 Let $(\A, T)$ be a one-step cocycle generated by $(A_1, \ldots, A_k)\in SL(2, \R)^k$. Let the one-step cocycle $(\A,T)$ satisfying in pinching and twisting conditions. By Theorem \ref{mather:dom}, the cocycle $\A$ is dominated with respect to the Mather set $\mathcal{K}$, so there are continuous maps $e_1$ and $e_2$ from $\mathcal{K}$ to $\mathbb{P}\mathbb{R}^2$. Then, we show that $h_{\mu}(T)=0$ for all $\mu$ supported in $\mathcal{K}$ under the injectivity of $e_1$ and $e_2$ on $\mathcal{K}$.
A distinction in our result from Bochi and Rams' result \cite{BR} is that we consider typical cocycles, which are open and dense, as our assumptions, while they considered dominated cocycles, which hold only for a restrictive family of matrices, to obtain the zero entropy of Lyapunov maximizing measures. Also, we only assume the injectivity condition on the Mather set, not on the whole \(\Sigma\).

  The main result of this paper is the following:
\begin{thm}\label{main_th}
Assume that $(A_1, \ldots, A_k)\in SL(2, \R)^k$ generated a one-step cocycle $\A:\Sig \to SL(2, \R)$. Let the one-step cocycle $\A:\Sig \to SL(2, \R)$ satisfying pinching and twisting conditions. Suppose that the continuous maps $e_1$ and $e_2$ are one-to-one on $\mathcal{K}$. Then the entropy of any Lyapunov maximizing measure is zero.
\end{thm}

 Note that the injectivity condition is indeed required to show that \(h_{\mu}(T) = 0\) for all \(\mu\) supported in \(\mathcal{K}\). We give an example of a typical cocycle where the injectivity condition fails.
\begin{ex}\label{example}
Let $\A$ be a one-step cocycle generated by  matrices $A_1:=\left[\begin{array}{cc}
\lambda & 0 \\
0 & \frac{1}{\lambda}
\end{array}\right]$, where $\lambda >1$, $A_{2}:=\left[\begin{array}{cc}
\lambda & 0 \\
1 & \frac{1}{\lambda}
\end{array}\right]$ and an irrational rotation matrix $A_3:=R_{\theta}$. By definitions, the cocycle $\A$ generated by $(A_1, A_2, A_3)$ is a typical cocycle and is not a uniformly hyperbolic. Moreover, $\beta(\A)=\log \lambda.$

The cocycle $\A$ is dominated with respect to its Mather set (see Theorem \ref{mather:dom}). Thus, there
are continuous map $e_1$ and $e_2$ from $\mathcal{K}$ to $\mathbb{P}\R^2$. By our contraction, the set $\{1, 2\}^{\Z}$ is
contained in the Mather set $\mathcal{K}$.

It is easy to see that \((A_1, A_2)\) satisfies the forward NOC but not the backward NOC. Indeed, the actions of \(A_1^{-1}\) and \(A_2^{-1}\) on the projective space \(\mathbb{P}\mathbb{R}^2\) have a common fixed point, which is the line spanned by \(\left[\begin{array}{c} 0 \\ 1 \end{array}\right]\). Hence, the invariant multicone \(M\) for \(\{A_1^{-1}, A_2^{-1}\}\) must contain this line, which implies that \(A_1^{-1} M \cap A_2^{-1} M \neq \varnothing\). Thus, the injectivity condition fails (see \cite[Proposition 2.10]{BR}). Moreover, there are Lyapunov maximizing measures with positive entropy.

 \end{ex}
 
 In the following example, we provide an example of application of the main result that is
outside the scope of \cite{BR}.
\begin{ex}\label{example1}
Let $\A$ be a one-step cocycle generated by matrices $A_1:=\left[\begin{array}{cc}
\lambda & 0 \\
0 & \frac{1}{\lambda}
\end{array}\right]$, where $\lambda > 1$, and an irrational rotation matrix $A_2:=R_{\theta}$. By definition, the cocycle $\A$ generated by $(A_1, A_2)$ is a typical cocycle and is not uniformly hyperbolic. Moreover, $\beta(\A)=\log \lambda.$
It is easy to see that the point \(1^{\infty}\) is in the Mather set. We claim that the Mather set \(\mathcal{K}=\{1^{\infty}\}\). Indeed, the support of a measure is the intersection of the compact sets of full measure by definition. If \(\mu\) is any ergodic measure that is not supported on the point of all 1's, then \(\mu([2])>0\). In this case, by using the submultiplicativity of the norm,

$$
\lim_{n \rightarrow \infty} \frac{1}{n} \log \left\|A_{i_{n}} \ldots A_{i_1}\right\| \leq \log \left\|A_1\right\|^{\mu([1])} + \log \left\|A_2\right\|^{\mu([2])} \leq \mu([1]) \log \lambda.
$$

This means the only measure achieving the maximal  Lyapunov exponent $\log \lambda$ is the \((\delta_1)^{\mathbb{Z}}\)-measure. So, \(\mathcal{K}=\{1^{\infty}\}\). Therefore, the injectivity condition holds, and the entropy of the Lyapunov maximizing measure is zero.

\end{ex}

Note that one can also modify $A_1$ in Example \ref{example} such that $(A_1, A_2)$ satisfies the NOC condition in Example \ref{example}. Then, similar to the argument of Example \ref{example1}, one can show that the injectivity condition holds, so the entropy of any Lyapunov maximizing measure is zero.

This paper is organized as follows.  In Section 2 we prove the main theorem. In Section 3 we discuss the relation between the ergodic optimization of Lyapunov exponents and the ergodic optimization of Birkhoff averages. Let us mention that some of the key parts of the paper inspired by
ideas from the papers of Bousch and Mairesee \cite{BMa} and Bochi and Rams \cite{BR}.

\section{Proof of the main theorem}\label{section3}

\subsection{Barabanov functions}
Assume that $(A_1, \ldots, A_k) \in GL(d,\R)^k$ generated a one-step cocycle $\A: \Sig \to GL(d,\R).$  We recall that we denote by $\beta(\A)$ the maximal Lyapunov exponent. We say that the cocycle $\A$ is irreducible if there is no non-zero proper linear subspace $V$ such that $A_{i} V \subset V.$ Barabanov \cite{Bar} proved that there exists an \textit{extremal norm} $||| . |||$ on $\R^d$ with the following stronger property:

\[\forall u \in \mathbb{R}^{d}, \quad \max _{i \in\{1, \ldots, k\}}|||A_{i} u |||=e^{\beta(\A)}|||u|||, \]

whenever $A$ is irreducible. The norm can be seen as arising from a non-commutative
version of Mañé Lemma. For more information on the Barabanov norms, see \cite{J} and \cite{W}.

Bochi and Garibaldi \cite[Theorem 5.7]{BGar} established the existence of extremal norms in a far
more general setting. In particular, if the one-step cocycle  cocycle $(\A, T)$ satisfies  pinching and twisting  conditions, then 
\begin{equation}\label{Barb}
\forall u \in \mathbb{R}^{d}, \quad \max _{i \in\{1, \ldots, k\}}|||A_{i} u |||=e^{\beta(\A)}|||u|||.
\end{equation}

We fix  a typical one-step cocycle $\A: \Sigma \to \sl2$ with generator $(A_1, \ldots, A_k)$. We recall that the cocycle $\A$ is dominated with respect to the Mather set $\mathcal{K}$.

 We can define the set of optimal future trajectories by using \eqref{Barb} that is
\[J=\left\{(\omega, v)\in \mathcal{K} \times \mathbb{P}\R^2: \log |||\A^{n}(\omega) v |||=n\beta(\A)+\log|||v|||\right\}.\]
By \cite[Proposition 6.2 \& Theorem 6.5]{BGar}, the following holds:
\begin{equation}\label{matherset belongs to J}
\left(\omega, e_1(\omega)\right) \in J  \text{ for all } \omega \in \mathcal{K}. 
\end{equation}

\begin{lem}\label{geom}
If $(\omega, v)\in J$ and $u \in \mathbb{P}\R^2$ are such that $v-u \in e_{2}(\omega)$ then
\[|||v||| \leq |||u|||.\]

\end{lem}
\begin{proof}
Suppose that $\omega \in \mathcal{K}$ and $v, u \in \mathbb{P}\R^2$ be such that $v-u \in e_{2}\left(\omega\right) .$ Assume that $v_{n}:=\A^{n}\left(\omega\right) v$
and $u_{n}:=\A^{n}\left(\omega\right) u$, for $n \geqslant 0$. Since $v, u \notin e_{2}\left(\omega\right)$, it follows from \eqref{domin} that the quantities
\begin{equation}\label{dom:prop}
\frac{\left\|v_{n}-u_{n}\right\|}{\left\|v_{n}\right\|}\text { and } \frac{\left\|v_{n}-u_{n}\right\|}{\left\|u_{n}\right\|} \text { tend to } 0 \text { as } n \rightarrow \infty .
\end{equation}
We are going to show that
\begin{equation}\label{minus}
\lim_{n\to \infty}|\log |||u_{n}|||-\log |||v_{n}||| |=0.
\end{equation}

Norms are equivalent; there is $C>0$ such that
\begin{equation}\label{equivalen:norms}
C^{-1} ||w|| \leq |||w||| \leq C ||w||.
\end{equation}

Now, \eqref{minus} can be estimated
$$\begin{aligned}
|\log |||u_{n}|||-\log |||v_{n}||| |&\leq \max\bigg(\frac{|||u_{n}|||}{|||v_{n}|||}-1, \frac{|||v_{n}|||}{|||u_{n}|||}-1 \bigg)\\
& \leq \frac{|||v_{n}-u_{n}|||}{\min(|||u_{n}|||, |||v_{n}|||)}\\
&\stackrel{\eqref{equivalen:norms}}{\leq}C^2\frac{||v_{n}-u_{n}||}{\min(||u_{n}||, ||v_{n}||)},
\end{aligned}
$$
which by \eqref{dom:prop} goes to zero as well. This proves \eqref{minus}.

Since $(\omega, v) \in J$, for all $n\geq 0$,
\[\log |||v_n|||=n\beta(\A)+\log |||v|||.\]
By \eqref{Barb},
\[\log |||u_n|||\leq n\beta(\A)+\log |||u|||.\]

In particular,
\[\log |||u_n||| -\log |||v_n||| \leq \log |||u|||-\log |||v|||.\]
Taking limits as $n \to \infty.$
\end{proof}

If $A,B,C,D$ are distinct points in the $\R_{*}^2$, then we define their
\textit{cross ratio} to be
\[ [A,B;C,D]:=\frac{A \times C}{A\times D}. \frac{B\times D}{B \times C},\]
where $\times$ denotes cross-product in $\R^2$, i.e., determinant. Furthermore, the cross-ratio is invariant under linear transformations.

Now, we use Lemma \ref{geom} to prove the following Lemma.
\begin{thm}\label{cross-ratio property}
If $(x,v_1), (y,w_1) \in J$ and non-zero vectors $v_{2} \in e_{2}(x), w_{2} \in e_{2}(y)$, then 
\[|[v_{1}, w_{1}; v_{2},w_{2}]|\geq 1.\]

\end{thm}
\begin{proof}
Since $e_{1}$ direction is different from any $e_{2}$ direction (see \cite[Proposition 2.7]{BR}),  $v_1$ or $w_1$ can not be collinear to $v_2$ or $w_2$, so the  cross-ratio
is well defined. Furthermore, one can write
\[v_1=c_1 v_2 +c_2 w_1 \quad \text{and} \quad w_1=\beta_1 w_2+\beta_2 v_1 .\]
By Lemma \ref{geom},
\[ |||v_1 ||| \leq |||c_{2}w_{1}|||\leq |||c_{2} \beta_{2} v_{1}|||=|c_{2}\beta_{2}| |||v_{1}|||.\]

Hence, $|c_2 \beta_2| \geq 1.$ We substitute
\[c_2=\frac{v_1 \times v_2}{w_1 \times v_2} \quad \text{and} \quad \beta_{2}=\frac{w_1 \times w_2}{v_1 \times w_2}.\]
The assertion is obtained.
\end{proof}

By Theorem \ref{cross-ratio property} and \eqref{matherset belongs to J}, we obtain the following corollary.
\begin{cor}\label{cross-ratio-1}
Assume that $x,y \in \mathcal{K}$. Then,
\[ |[e_{1}(x), e_{1}(y); e_{2}(x), e_{2}(y)]|\geq 1. \]
\end{cor}

Let $\left(v_{1}, w_{1} ; v_{2}, w_{2}\right)$ be a 4-tuple of distinct points in $\mathbb{P}\R^2$. Then, we only have the following cases:
\begin{itemize}
\item antiparallel configuration: \[v_{1}<w_{2}<w_{1}<v_{2}<v_{1}\] for some cyclic order $<$ on $\mathbb{P}\R^2$;
\item  coparallel configuration: \[v_{1}<w_{1}<w_{2}<v_{2}<v_{1}\] for some cyclic order $<$ on $\mathbb{P}\R^2$;
\item  crossing configuration: $v_{1}<w_{1}<v_{2}<w_{2}<v_{1}$ for some cyclic order $<$ on $\mathbb{P}\R^2$;

\end{itemize}
see \cite[Figures 2,3 and 4]{BR}.

Let $v_{1}, v_{2}$ be points in the unit circle $\partial \mathbb{D}$, where $\mathbb{D}$ is a disk, and let $\overrightarrow{v_{2} v_{1}}$ be the oriented hyperbolic geodesic from $v_{2}$ to $v_{1}$. We identify $\partial \mathbb{D}$ with the projective space $\mathbb{P}\R^2$ as follows:
$e^{2 \theta i} \in \partial \mathbb{D} \leftrightarrow(\cos \theta, \sin \theta)^{\prime} \in \mathbb{P}\R^2$.

We say that two geodesics $\overrightarrow{v_2 v_1}$ and $\overrightarrow{w_{2} w_{1}}$ with distinct endpoints are antiparallel, coparallel or crossing according to the configuration  $\left(v_{1}, w_{1} ; v_{2}, w_{2}\right)$.

We define the following compact subset $\mathbb{P}\R^2 \times \mathbb{P}\R^2:$
\[\Gamma:=\{(e_1(x), e_2(x)): x\in \mathcal{K}\}.\]

The set $\Gamma$ can be decomposed it into fibers in two different
ways:
\[ \Gamma=\bigcup_{\omega_1 \in e_{1}(\mathcal{K})} \{ \omega_1\} \times \Gamma_{2}(\omega_1)=\bigcup_{\omega_2 \in e_{2}(\mathcal{K})} \Gamma_1 (\omega_2) \times \{\omega_2\}.\]

\begin{thm}\label{Not coparallel}
Suppose that $(v_1, v_2), (w_1, w_2)\in \Gamma$. Then, $(v_1, w_1; v_2, w_2)$ cannot be in coparallel configuration.
\end{thm}
\begin{proof}
Suppose that $v_{1}, v_2, w_1, w_2$ are distinct; otherwise, there is nothing to prove. By Corollary \ref{cross-ratio-1}, $|[v_1, w_1; v_2, w_2]|\geq 1$, so $[v_1, w_1; v_2, w_2]\leq -1$ or $[v_1, w_1; v_2, w_2]\geq 1$. On the other hand, the configuration $(v_1, w_1; v_2, w_2)$ is coparallel if and only if $0<[v_1, w_1; v_2, w_2]<1$, by \cite[Proposition 4.4]{BR}. Thus, the configuration can not be coparallel.
\end{proof}

Now, we are going to show that each direction $e_1(\omega)$ or $e_2(\omega)$ uniquely establishes the other, except for a countable number of bad directions for each $\omega \in \mathcal{K}.$ We define the following sets:

\[\begin{aligned}
&N_{1}:=\left\{v_{1} \in e_{1}(\mathcal{K}) ; \Gamma_{2}\left(v_{1}\right) \text { has more than one element}\right\}, \\
&N_{2}:=\left\{v_{2} \in e_{2}(\mathcal{K}) ; \Gamma_{1}\left(v_{2}\right) \text { has more than one element}\right\}.
\end{aligned}\]

We apply the same arguments used in \cite{BR}.
\begin{prop}\label{countable}
The $N_i$ is countable for $i\in \{1, 2\}.$
\end{prop}

\begin{proof}
We will only consider $i=1$ as the proof of the other case is similar.

We denote by $I_{1}(v_1)$ the least closed subinterval of $\mathbb{P}\R^2 \setminus \{v_1\}$ containing $\Gamma_2(v_1)$ for each $v_1 \in N_1.$ We are going to show that  $I_{1}(v_1)$, $I_{1}(v_2)$ have disjoint interiors in $\mathbb{P}\R^2$ if $v_1, v_2 \in N_{1}$ are distinct.

We denote by $v$ and $w$ the endpoints of the interval $I_1(v_1)$ and take any point $v_3$ in its interior. Thus, $\overrightarrow{v_3 v_2}$ is coparallel to one of the two geodesics $\overrightarrow{vv_1}$ or $\overrightarrow{wv_1}.$ By Theorem \ref{Not coparallel}, one concludes that $(v_2, v_3)\notin \Gamma$ as $(v_1, v)$ and $(v_1, w)$ in $\Gamma.$ Thus, $\Gamma_2(v_2) \cap \mathring{I_1} (v_1)=\emptyset,$ and, in particular, $\partial I_{1}(v_2) \cap \mathring{I_1} (v_1)=\emptyset.$ Similarly, $\partial I_{1}(v_1) \cap \mathring{I_1} (v_2)=\emptyset.$ Thus,  $\mathring{I_1} (v_1)\cap \mathring{I_1} (v_2)=\emptyset.$ 

 It follows from separability of the circle that $N_1$ is countable.
\end{proof}

Therefore, the following implication holds:
$$
\left.\begin{array}{l}
x, y \in \mathcal{K} \\
e_{i}(x)=e_{i}(y) \notin N_{i} \text { for some } i
\end{array}\right\} \Rightarrow\left\{\begin{array}{l}
e_{1}(x)=e_{1}(y) \\
e_{2}(x)=e_{2}(y)
\end{array}\right.
$$

Assume that $\mu$ is a Lyapunov maximizing measure that is not atomic; otherwise, there is nothing to prove. Now, we are going to use the fact that $e_{1}$ is one-to-one on $\mathcal{K}.$ Let $\pi_{+}: \omega \mapsto \pi_{+}(\omega)=(\omega_{0}, \omega_{1}, \ldots).$

\begin{thm}
$\mu(e_{1}^{-1}(N_1))=0$
\end{thm}
\begin{proof}
 The set $N_{1} \subset \mathbb{P}\mathbb{R}^2$ is countable by Proposition \ref{countable}. The set $e_{1}^{-1}\left(N_{1}\right) \subset \mathcal{K}$ is a countable union of sets of the form $\left\{\omega_{-}\right\} \times \pi_{+}(\mathcal{K})$ as $e_1$ is one-to-one on $\mathcal{K}$. By contradiction, suppose  that $e_{1}^{-1}\left(N_{1}\right)$ has positive measure. Then there exists $\omega_{-}$ such that $U:=\left\{\omega_{-}\right\} \times \pi_{+}(\mathcal{K})$ has positive measure. It follows from Poincaré recurrence that there exists $p \geqslant 1$ such that $T^{-p}(U) \cap U \neq \varnothing$. Therefore, the infinite word $\omega_{-}$ is periodic with period $p$. This implies that $T^{-p}(U) \subset U$. By invariance, $\mu\left(U \backslash T^{-p}(U)\right)=0$ and
$$
\begin{aligned}
\mu\left(\bigcap_{n \geqslant 0} T^{-n p}(U)\right) &=\mu(U)-\mu\left(U \backslash T^{-p}(U)\right)-\mu\left(T^{-p}(U) \backslash T^{-2 p}(U)\right)-\cdots \\
&=\mu(U)>0 .
\end{aligned}
$$
On the other hand, the set $\bigcap_{n \geqslant 0} T^{-n p}(U)$ is a singleton, but we assumed that $\mu$ is non-atomic. 
\end{proof}

Therefore each of the directions $e_1$ and $e_2$ uniquely determines the other $\mu$-a.e. by the definitions of $N_1$ and $\Gamma$. To be more precise, by \cite[Lemma 5.1]{BR}, there is a Borel measurable function $g$ such that
\[ g \circ e_1=e_2 \text{  for } \mu\text{-almost all},\]
where $\mu$ is a Lyapunov maximizing measure. Then, for $\mu$-almost every $\omega \in \mathcal{K},$ the past $\omega_{-}$ uniquely determines the future $\omega_{+};$
\[ \omega_{-} \xrightarrow{\hspace{0.2cm}   e_1 \hspace{0.2cm}    } e_{1}(\omega_{-}) \xrightarrow{\hspace{0.25cm} g\hspace{0.25cm} } e_{2}(\omega_{+}) \xrightarrow[\text{one-to-one}]{e_{2}^{-1}} \omega_{+}.\]

This implies $h_{\mu}(T)=0$, for all $\mu$-supported in $\mathcal{K}.$

We end this section by commenting on \cite{BGar}, from which we use some results.

\begin{rem}
Bochi and Garibaldi \cite{BGar} introduced the {\it uniform spannability} property and proved several results that we used some of them in this article under this property. They also showed that pinching and twisting conditions imply uniform spannability (see \cite[Remark 3.12]{BGar}). Recently, Mohammadpour and Park \cite[Theorem 1.1]{MP-uniform-qm} showed that a similar version of the uniform spannability property holds for one-step cocycles under the irreducibility assumption (see \cite[Remark 2.6]{MP-uniform-qm}).
\end{rem}

\section{Ergodic optimization of Lyapunov exponents of typical cocycles and Birkhoff averages are equivalent}
Let $\A: \Sigma \to \sl2$ be a typical cocycle. We are going to show that there is a continuous function $f \in C(\Sigma)$ such that $\mathcal{M}_{max}(\A)=\mathcal{M}_{max}(f)$.  Moreover, our result provides a relation between the Birkhoff sum $S_{n}f$ and the norm of the matrix product $\A^{n}.$

We say that $\Phi:=\{\log \phi_{n}\}_{n=1}^{\infty}$ is an almost additive potential over a TDS $(X, T)$  if there exists a constant $C > 0$ such that for any $m,n \in \N$, $x\in X$, we have
\[ C^{-1}\phi_{n}(x)\phi_{m}(T^{n})(x) \leq \phi_{n+m}(x)\leq C \phi_{n}(x) \phi_{m}(T^{n}(x)).\]

We say that $\A:X\rightarrow GL(d, \R)$ is \textit{almost multiplicative} over a TDS $(X, T)$ if there is a constant $C>0$ such that
\[||\A^{m+n}(x)|| \geq C ||\A^m(x)|| ||\A^n(T^m(x))|| \hspace{0.2cm}\forall x\in X, m,n\in \N.\]

We note that since clearly $||\A^{m+n}(x)|| \leq  ||\A^m(x)|| ||\A^n(T^m(x))|| \hspace{0.1cm}$ for all $x\in X, m,n\in \N$, the condition of almost multiplicativity of $A$ is equivalent to the statement that $\Phi_{\A}=\{\log \|\A^{n}\|\}_{n=1}^{\infty}$ is almost additive.

The author \cite{Moh20} showed that if a cocycle is dominated with index 1, which can be  characterized in terms of existence of invariant cone fields (or multicones)\cite{ABY, BG09}, then the potential $\Phi_{\A}$ is almost additive.

\begin{thm}\label{almost_additive-dom}
Let $X$ be a compact metric space, and let $T:X \to X$ be a homeomorphism. Assume that the cocycle $\A: X\rightarrow GL(2, \R)$ is uniformly hyperbolic  over  $(X, T)$.  Then, there exists $\kappa>0$ such that for every $m,n>0$ and for every $x\in X$ we have

\[
||\A^{m+n}(x)|| \geq \kappa ||\A^m(x)|| \cdot ||\A^n(T^m(x))||.
\]
\end{thm}
\begin{proof}
It follows from \cite[Proposition 5.8]{Moh20}.
\end{proof}

Cuneo \cite{Cu} showed that  every almost additive potential sequence is actually equivalent to an additive potential in the sense that there exists a continuous potential with the same equilibrium states, topological pressure,  weak Gibbs measures, variational principle, level sets (and irregular set) for the Lyapunov exponent.

\begin{thm}[{{\cite[Theorem 1.2]{Cu}}}]\label{approx}
Let $\Phi=\{\log \phi_{n}\}_{n=1}^{\infty}$ be an almost additive sequence of potentials over the topological dynamical systems $(X,T)$. Then, there exists $f \in C(X)$ such that
\[
\lim _{n \rightarrow \infty} \frac{1}{n}\left\|\log \phi_{n}-S_{n} f\right\|=0
\]
\end{thm}

\begin{thm}\label{eq1}
Assume that $(A_1, \ldots, A_k) \in SL(2,\R)^k$ generated a one-step cocycle $\A: \Sig \to SL(2,\R).$ Let $\mathcal{A}:\Sig \to SL(2,\R)$ be a typical cocycle.  Then, there is a continuous function $f \in C(\mathcal{K})$ such that 
\[ \mathcal{M}_{max}(\A)=\mathcal{M}_{max}(f).\]
\end{thm}
\begin{proof}
First, we define the induced first-return cocycle on the Mather set, and then we use above theorems.

Since the cocycle $\A$ is dominated with respect to the Mather set $\mathcal{K}$ by Theorem \ref{mather:dom}, one can define the following dominated cocycle: Let $T_{\mathcal{K}}: \mathcal{K} \rightarrow \mathcal{K}$ be the first return map defined by
\[
T_{\mathcal{K}}(x):=T^{N_{\mathcal{K}}(x)}(x) \quad \text { where } N_{\mathcal{K}}(x):=\inf \left\{n \geqslant 1: T^{n}(x) \in \mathcal{K}\right\} .
\]

Let  
$\B:=  \mathcal{K} \rightarrow SL(2, \mathbb{R})$ be the function defined by $\B(x)=\A^{N_{\mathcal{K}}(x)}(x)$\footnotemark \footnotetext{Note that $\chi(x, \A)=\chi(x, \B)$ for any $\mu$-ergodic Lyapunov maximizing measure.}. By Theorem \ref{almost_additive-dom}, $\Phi_{\B}=\{\log \|\B^{n}\|\}_{n=1}^{\infty}$ is almost additive. Therefore, there is a continuous function $f$ such that
\begin{equation}\label{same}
\lim _{n \rightarrow \infty} \frac{1}{n}\left\|\log \|\B^{n}\|-S_{n} f\right\|=0,
\end{equation}
by Theorem \ref{approx}. Thus, \[\chi(\mu, \B)=\int f d\mu,\]
for any maximizing measure $\mu.$ 
\end{proof}

 \subsection*{Acknowledgements}
The author is grateful to Philippe Thieullen for his careful reading of the paper and useful comments. He would also like to thank Rafael Potrie and Michal Rams for useful discussions and Adam Abrams for helping him with the figure. Finally, he would like to express his gratitude to the anonymous referee for their valuable corrections and suggestions, which greatly contributed to the improvement of the paper.

This work was supported by the Agence Nationale de la Recherche through the project Codys (ANR-18-CE40-0007) and the Knut and
Alice Wallenberg Foundation.
\bibliographystyle{acm}
\bibliography{zero-entropy}
\end{document}